\documentclass[12pt]{amsart}

\usepackage{amsmath}
\usepackage{amssymb, amsfonts,amscd,verbatim}
\usepackage{hyperref}
\usepackage{mathtools}
\usepackage{graphicx}
\usepackage{cleveref}
\usepackage{enumitem}
\usepackage[all]{xy}
\usepackage{tikz}
\usetikzlibrary{matrix}

\usepackage{booktabs}

\newtheorem{thm}{Theorem}[section]
\newtheorem{prop}[thm]{Proposition}
\newtheorem{cor}[thm]{Corollary}
\newtheorem{lemma}[thm]{Lemma}
\newtheorem{defn}[thm]{Definition}
\newtheorem{rem}[thm]{Remark}
\numberwithin{equation}{section}

\newtheorem{example}[thm]{Example}

\renewcommand{\k}{\ensuremath{\mathbf{k}}}
\newcommand{\xkn}{\ensuremath{X(\mathbf{k}^n)}}

\def\Z{{\mathbb Z}}
\def\N{{\mathbb N}}
\def\F{{\mathbb F}}

\DeclareMathOperator{\spn}{span}

\DeclareMathOperator{\rank}{rank}

\DeclareMathOperator{\Lin}{Lin}
\DeclareMathOperator{\im}{Im}

\begin{document}

\title{Mod $p$ Buchstaber invariant}

\author{Djordje Barali\'{c}, Ale\v s Vavpeti\v c, Aleksandar Vu\v{c}i\'{c}}
\address{ \scriptsize{Mathematical Institute SASA, Belgrade, Serbia }}
\email{djbaralic@mi.sanu.ac.rs}
\address{Faculty of Mathematics and Physics, University of Ljubljana, Slovenia and Institute of Mathematics, Physics and Mechanics, Ljubljana, Slovenia}
\email{ales.vavpetic@fmf.uni-lj.si}
\address{\scriptsize{University of Belgrade, Faculty of Mathematics, Belgrade, Serbia}}
\email{avucic@matf.bg.ac.rs}






\maketitle

\begin{abstract}
We investigate the mod $p$ Buchstaber invariant of the skeleta of simplices, for a prime number $p$, and compare them for different values of $p$. For $p=2$, the invariant is the real Buchstaber invariant.  Our findings reveal that these values are generally distinct. Additionally, we determine or estimate the mod $p$ Buchstaber invariants of certain universal simplicial complexes $X(\F_p^n)$.
\end{abstract}

\section{Introduction}

Moment angle complexes appeared in the papers of Buchstaber and Panov \cite{BucPan98} and \cite{BucPan99} as cell decompositions of certain toric spaces in algebraic geometry, combinatorial topology, and symplectic geometry. They are closely related to several important constructions in mathematics. Among them are intersections of Hermitian quadrics \cite{10.1007/s11511-006-0008-2} from holomorphic dynamics, Hamiltonian toric manifolds \cite{Audin} and complements of arrangements of coordinate spaces in a complex space \cite{Longueville}. Recently, they have found applications in the theory of mechanical linkages or `arachnoid mechanisms' \cite{Izmestev_2000} and topological data analysis \cite{10.1093/bib/bbad046}.

Moment angle complexes, their quotients, and various generalizations are extensively studied in toric topology. Construction of moment angle complex $\mathcal{Z}_\mathcal{K}$ induces a functor from the category of simplicial complexes and simplicial maps to the category of spaces with toric actions and equivariant map. Determination of the free toral rank of $\mathcal{Z}_\mathcal{K}$ is a very subtle problem since it is a combinatorial characteristic of $\mathcal{K}$, known in the literature as the \textit{Buchstaber invariant} or the \textit{Buchstaber number}.  The problem turned out difficult and only partial results are obtained up to now. The exact value of the Buchstaber invariant for graphs was obtained by Ayzenberg in \cite[Theorem~1]{Ayze1}. Erokhovets found a criterion for simplicial complexes to have Buchstaber number equal to zero, one and two. In \cite{MR1897064} Buchstaber and Panov showed a lower bound for the Buchstaber invariants in terms of the chromatic number of a simplicial complex. A generalized chromatic number method for estimation of the Buchstaber invariant was developed in \cite{Ayze2} and a method based on minimal non-simplices in \cite{Erok} and \cite{Ayze}. Ayzenberg in \cite{Ayze} constructed two simplicial complexes whose Tor-algebras are isomorphic but have distinct Buchstaber invariants. More about the Buchstaber invariant of simple polytopes and simplicial complexes can be found in the survey \cite{MR3482595}.

In toric topology, real counterparts of standard torus actions are studied. Group $\F_2^m$ acts on the real moment angle complex $_\mathbb{R} \mathbb{Z}_\mathcal{K}$ where $m$ is the number of vertices of simplicial complex $\mathcal{K}$, but this action is not free unless $\mathcal{K}$ is simplex. Maximal rank of the subgroup of $\F_2^m$ that acts freely on $_\mathbb{R} \mathbb{Z}_\mathcal{K}$ is known as the \textit{real Buchstaber invariant of $\mathcal{K}$}. Buchstaber invariant of a simplicial complex is less or equal than its real Buchstaber number, and the equality case is closely related to the lifting problem in toric topology, posted by L\"{u}. The progress in the study of the real Buchstaber invariant was achieved by Fukukawa and Masuda, who considered the real Buchstaber number for the skeletons of simplex $\Delta^n_k$ in \cite{FukMas} and reformulated the problem into a linear programming problem. The Buchstaber invariants of the skeletons of simplex are of particular interest because they appear in the estimations of the Buchstaber invariants of a simplicial complex, see \cite{MR3482595}. Cho in \cite{10.2969/jmsj/06841695} completed the calculations of the real Buchstaber invariants of $\Delta^n_k$.

An alternative approach to the Buchstaber invariants is based on nondegenerate maps into so called \textit{universal simplicial complexes}. They were introduced by Davis and Januszkiewicz in \cite{DJ} as simplicial complexes whose maximal simplices are bases of the lattices of $\mathbb{Z}^n$ or $\F_2^n$. They are closely related to the fundamental geometric object known as Tits building studied in homological algebra, cohomology of groups, number theory, etc.  The importance of universal complexes in toric topology is in their role in generalizing toric varieties and the study of the Buchstaber invariants. Indeed, for a finite simplicial complex, its complex and real Buchstaber invariants are equal to the difference of the number of its vertices and the minimal $n$  such that there exists a nondegenerate simplicial map from the simplicial complex into the universal complexes of the bases of the lattices in $\mathbb{Z}^n$ and $\F_2^n$, respectively.  This approach, developed by Ayzenberg in \cite{Ayze2}, motivated Barali\'{c}, Grbi\'{c}, Vavpeti\v{c} and Vu\v{c}i\'{c} in \cite{BGVV} to define \textit{the mod $p$ Buchstaber invariant} based on the universal complex of bases of the lattice of  $\F_p^n$ for any prime number $p$.

Determination of the Buchstaber invariants of the universal complexes is a question of essential interest in toric topology. These values also appear in the bounds of the Buchstaber numbers of a simplicial complex $K$ and in the obstructions for the lifting problem. In \cite{Sun2017}, Sun proved that the complex Buchstaber number of the real 3-dimensional universal complex is 10, and Shen in \cite{QShen} showed that the complex Buchstaber number of the real 4-dimensional universal complex is no less than 24.

In the article, we study the mod $p$ Buchstaber invariant of the skeletons of a simplex and some relations between various mod $p$ Buchstaber invariants of a simplicial complex. In Section 2, for any prime $p$ we introduce the mod $p$ Buchstaber invariant of a simplicial complex based on the universal complex bases of the lattice of  $\F_p^n$. Estimations and calculations of the mod $p$ Buchstaber invariant of a skeleton of a simplex are studied in Section 3. In Section 4, we present another approach to the real Buchstaber invariant of a skeleton of a simplex based on embeddings into the real universal complexes. The mod $p$ Buchstaber invariant of the real 3-dimensional universal complex is calculated for any prime number $p$.

\section{Buchstaber invariant}

Let $\mathcal{K}$ be an $n$-dimensional simplicial complex with the set of vertices $[m]$. Let
us denote by $(\underline{X}, \underline{A})=\{(X_i,
A_i)\}^m_{i=1}$ a collection of topological pairs of CW-complexes.
The polyhedral $\mathcal{K}$ product is a topological space $\mathcal{Z}_\mathcal{K}
(\underline{X}, \underline{A})=\bigcup_{\sigma \in \mathcal{K}} D(\sigma)$
where
$$D(\sigma)=\prod_{i=1}^m Y_i, \qquad\mbox{and}\qquad
Y_i=
\begin{cases}
X_i &\text{if } i\in\sigma,\\
A_i &\text{if } i\not\in\sigma.
\end{cases}
$$
By definition we have $D(\emptyset)=A_1\times \dots\times A_m$.

In the case $(\underline{X}, \underline{A})=\{(D^2,
S^1)\}^m_{i=1}$, $\mathcal{Z}_\mathcal{K} (\underline{X}, \underline{A})$ is denoted by $\mathcal{Z}_K$ and it is called \textit{the moment angle complex} of $\mathcal{K}$ and in the case $(\underline{X}, \underline{A})=\{(D^1, S^0)\}^m_{i=1}$, $\mathcal{Z}_\mathcal{K} (\underline{X}, \underline{A})$ is denoted by $_{\mathbb{R}} \mathcal{Z}_\mathcal{K}$ and it is called \textit{the real moment angle complex} of $\mathcal{K}$.

In the study of the moment angle and the real moment angle complexes, a special significance have their geometrical realization induced from the following models $D^2=\left\{z\in \mathbb{C} | |z|\leq 1\right\}$, $D^1=\left\{t\in \mathbb{R} | |z|\leq 1\right\}$, $S^1=\left\{z\in \mathbb{C} | |z|= 1\right\}$,$S^0=\left\{-1, 1\right\}$. The torus $T^m$ acts naturally coordinatewisely on $\mathcal{Z}_\mathcal{K}$ while $_{\mathbb{R}} \mathcal{Z}_\mathcal{K}$ has the natural coordinatewise action of the real torus $\F_2^m$. The diagonal actions of $S^1$ and $S^0$ on  $\mathcal{Z}_\mathcal{K}$ and $_{\mathbb{R}} \mathcal{Z}_\mathcal{K}$ are free. It is of particular interest to find the maximal dimension of subgroups of $T^m$ and $\F_2^m$ that act freely on $\mathcal{Z}_\mathcal{K}$ and $_{\mathbb{R}} \mathcal{Z}_\mathcal{K}$, respectively.

\begin{defn}\label{bi} A \textit{complex Buchstaber invariant} $s(\mathcal{K})$ of $\mathcal{K}$ is a maximal dimension of a toric subgroup of $T^m$ acting freely on $\mathcal{Z}_\mathcal{K}$.

A \textit{real Buchstaber invariant} $s_{\mathbb{R}} (\mathcal{K})$ of $\mathcal{K}$ is a maximal rank of a subgroup of $\F_2^m$ acting freely on $_{\mathbb{R}} \mathcal{Z}_\mathcal{K}$.
\end{defn}

In their seminal paper \cite{DJ}, Davis and Januszkiewicz first realized the importance of simplicial complexes whose maximal simplices are bases of the lattices of $\mathbb{Z}^n$ or $\F_2^n$ in generalizing toric varieties. These complexes, called universal complexes, are closely related to  Tits building studied in homological algebra, cohomology of groups and number theory. For theory and applications of buildings, we refer the reader to monograph \cite{AbraBrown}.

Let $\k$ be the field $\mathbb{F}_p$
or the ring $\mathbb{Z}$.  A set $\{ v_1, \ldots , v_m \}$ of elements in $\k^n$ is called \emph{unimodular} if $\spn\{v_1,\ldots ,v_m\}$ is a direct summand of $\k^n$ of dimension $m$. Note that a subset of an unimodular set is itself unimodular.

\begin{defn}
 The universal simplicial complex $\xkn$ on unimodular vertices $v_i\in \k^n$ is a simplicial complex whose simplices are all unimodular subsets of $\k^n$.
\end{defn}

The simplicial complexes $X (\F_p^n)$ are $(n-1)$-dimensional matroids and have the homotopy type of the wedge of spheres. In toric topology, $X (\Z^n)$ and  $X (\F_2^n)$ play important role in  combinatorial description of $s(\mathcal{K})$ and $s_{\mathbb{R}} (\mathcal{K})$. Following~\cite[Section 2]{Ayze1}, an equivalent way to introduce the Buchstaber invariants is the following;

\begin{prop}\label{bin} A \textit{complex Buchstaber invariant}  of $\mathcal{K}$ is the integer $s(\mathcal{K})=m-r$, where $r$ is the least integer such that there is a nondegenerate simplicial map
\[
f\colon \mathcal{K} \rightarrow X
	(\mathbf{\mathbb{Z}}^r) .
	\] 	
A \textit{real Buchstaber invariant}  of $\mathcal{K}$ is the integer $s_2 (\mathcal{K})=m-r$, where $r$ is the least integer such that there is a nondegenerate simplicial map
	$$ f\colon \mathcal{K} \rightarrow X (\F_2^r) .
	$$
\end{prop}

Ayzenberg in \cite{Ayze1} used this fact to define the Buchstaber number of an increasing sequence of simplicial complexes $\{L^i\}$. Recall that a sequence of simplicial complexes $\{L^i\}$, $i\in \mathbb{N}$ is called increasing if for all $j<k$ there exists a nondegenerate simplicial map $f\colon L^j\rightarrow L^k$. The Bushstaber  $\{L^i\}$ number  of a finite simplicial complex $\mathcal{K}$ is  the integer $s_{\{L^i\}} (\mathcal{K})=m-r$, where $r$ is the least integer such that there is a nondegenerate simplicial map $$ f\colon \mathcal{K} \rightarrow L^r.$$

The next proposition from \cite{Ayze1} is one of the main tools in determining and estimating various Buchstaber invariants.

\begin{prop}\cite[Proposition 2]{Ayze1} \label{ay1} Let $\{L^i\}$ and $\{M^i\}$ be two increasing sequences of simplicial complexes such that for each $i$ there is a nondegenerate map $f_i\colon L^i\rightarrow M^i$. Then $$s_{\{L^i\}} (\mathcal{K})\leq s_{\{M^i\}} (\mathcal{K}).$$
\end{prop}

For each prime number $p$ the universal complexes $X(\F_p^n)$ form increasing sequence of simplicial complexes, since $X(\F_p^i)$ embeds in $X(\F_p^j)$  if $i<j$. The mod $p$ map from $\Z^n$ to $\F_p^n$ induces nondegenerate map from $X(\Z^n)$ to $X(\F_p^n)$. These two observations motivated the authors of \cite{BGVV}  to define and study the mod $p$ Buchstaber number.

\begin{defn} A {\normalfont mod $p$ Buchstaber invariant}  of $\mathcal{K}$ is the integer $s_p (\mathcal{K})=m-r$, where $r$ is the least integer such that there is a nondegenerate simplicial map
	$$ f\colon \mathcal{K} \rightarrow X (\F_p^r) .
	$$
\end{defn}

Based on \cite[Lemma~7.32]{MR1897064}, a topological description of $s_p (\mathcal{K})$ can be given as a maximal rank of subgroup of $\F_p^m$  acting freely on  the polyhedral product $\mathcal{Z}_\mathcal{K} (\mathrm{Cone}( \F_p), \F_p)\subset D^{2m}$, where $\F_p$ is considered as the set of $p$-th roots of unity. Indeed, the statement and the proof of \cite[Lemma~7.32]{MR1897064} are analogous for $\mathcal{Z}_\mathcal{K} (\mathrm{Cone}(\F_p), \F_p)$ and subgroups of $\F_p^m$. Unlike in the cases of $\F_2^n$ and $\mathbb{Z}$, $\mathcal{Z}_\mathcal{K} (\mathrm{Cone}(\F_p), \F_p)$ for $p>2$ fails to be a manifold when $\mathcal{K}$ is a triangulation of a sphere.

We have a criterion to recognize the simplices from $\xkn$.
To vectors $u_0,\ldots, u_{k}$ of $\F_p^n$ we associate the $n\times (k+1)$ matrix $U=\Big( u_0\cdots u_k\Big)$ whose columns contain given vectors $u_i$, $0\leq i\leq k$.

\begin{prop} \cite[Proposition 1.5]{Baralic2023}
\label{GCD}
A set $\{u_0,\ldots ,u_k \}$ of vertices from $X(\F_p^n)$ is a $k$-simplex in $X(\F_p^n)$ if and only if  $\rank (U)= k+1$.
\end{prop}

We shortly write $\mathcal{X}^n=X(\F_p^n)$.
In what follows, we will need the description and the count of minimal nonsimplices of $\mathcal{X}^n$, which is given in the following lemma.

\begin{lemma} \cite[Lemma~1.10]{Baralic2023} \label{p nms2}
Let $p$ be a prime number, $n\geq 2$, $1\leq j\leq n$, and $v_0,\ldots , v_j$ vertices in $\mathcal{X}^n$.
\begin{itemize}
    \item [(a)]
    The set $\{ v_0,\ldots ,v_j\} $ is a minimal $j$-nonsimplex in $\mathcal{X}^n$ if and only if $\{v_0,\ldots ,v_{j-1}\}$ is a simplex in $\mathcal{X}^n$ and $v_j=\sum_{t=0}^{j-1}  a_tv_t$ for some $a_t\in \F_p\setminus \{0\}$.
    \item[(b)] For $j\geq 2$,
    \[
    |\{\text{minimal $j$-nonsimplices of } \mathcal{X}^n\}|=\frac{(p^n-1)(p^n-p)\cdots (p^n-p^{j-1})\cdot (p-1)^{j}}{(j+1)!}.
    \]
    \item[(c)] For $j=1$,
    \[
     |\{\text{minimal $1$-nonsimplices of } \mathcal{X}^n\}|=\frac{(p^n-1)(p-2)}{2}.
    \]
\end{itemize}
\end{lemma}
\begin{proof}
$(a)$ If $\{ v_0,\ldots ,v_j\}=:A$ is a minimal nonsimplex then, by definition, $\{v_0,\ldots,v_{j-1}\}$ is a simplex in $\mathcal{X}^n$ and $v_j$ is a linear combination of other vertices or otherwise $A$ would be a simplex in $\mathcal{X}^n$. Thus, $v_j=\sum_{t=0}^{j-1}  a_tv_t$, for some $a_t\in \F_p$. If, for example, $a_i=0$ then $\{v_0,\ldots ,v_{i-1},v_{i+1},\ldots ,v_j\}$ is not a simplex in $\mathcal{X}^n$ which contradicts the fact that $A$ is a minimal nonsimplex.

On the other hand, since $v_j=\sum_{t=0}^{j-1}  a_tv_t$ we get that $A$ is not a simplex in $\mathcal{X}^n$. We have given that $\{v_0,\ldots,v_{j-1}\}$ is a simplex in $\mathcal{X}^n$ and since all $a_t\neq 0$ we conclude that all other faces of $A$ are simplices in $\mathcal{X}^n$.

$(b)$ We have that the number of minimal $j$-nonsimplices is equal to the number of choices of a $(j-1)$-simplices multiplied by the number of choices for coefficients $a_t$ and divided by the number of faces of $A$. So we get
\[
\frac{(p^n-1)(p^n-p)\cdots (p^n-p^{j-1})}{j!}\cdot \frac{(p-1)^{j}}{j+1}
=\frac{(p^n-1)(p^n-p)\cdots (p^n-p^{j-1})\cdot (p-1)^{j}}{(j+1)!}.
\]

$(c)$ The number of minimal 1-nonsimplices equal to the difference between the number of pairs of vertices in $\mathcal{X}^n$ and the number of 1-simplices in $\mathcal{X}^n$, which is
\[
\frac{(p^n-1)(p^n-2)}{2}-\frac{(p^n-1)(p^n-p)}{2}=\frac{(p^n-1)(p-2)}{2}. \qedhere
\]
\end{proof}

Regarding nondegenerate simplicial maps into $\mathcal{X}^n$ we have the following interesting property. It shows that if there is one nondegenerate simplicial map into $\mathcal{X}^n$, then plenty exist.

\begin{prop}
\label{nondeg}
    If $f \colon \mathcal{K} \longrightarrow \mathcal{X}^n$ is a nondegenerate simplicial map, then so are all maps $g \colon \mathcal{K} \longrightarrow \mathcal{X}^n$ with $g(v)=\epsilon_v \cdot f(v)$, $\epsilon_v\in \{1, \ldots , p-1 \}$.
\end{prop}
\begin{proof}
    Let $\{v_0,\ldots , v_j\}$ be a $j$-simplex in $\mathcal{K}$. Then  $\{f(v_0),\ldots , f(v_j)\}$ is a $j$-simplex in $\mathcal{X}^n$. Therefore, by Proposition~\ref{GCD}, $\rank \Big( f(v_0) \cdots f(v_j) \Big)=j+1$. Now we have that
    \[
\rank \Big( g(v_0) \cdots g(v_j) \Big) =\rank \Big( f(v_1) \cdots f(v_j) \Big)=j+1
    \]
  and we conclude that $\{g(v_0),\ldots , g(v_j)\}$ is a $j$-simplex in $\mathcal{X}^n$.
\end{proof}

\section{Mod $p$ Buchstaber invariant of skeleta of simplices}

The skeletons of simplices are a pivotal class of simplicial complexes for determining the Bucshtaber number of an arbitrary simplicial complex. Fukukawa and Masuda studied the real Buchstaber invariant of the skeleton of simplex in \cite{FukMas}. As we mentioned in the introduction, Cho in \cite{10.2969/jmsj/06841695} completed the calculations of the real Buchstaber invariants of $\Delta^m_{(k)}$. The subsequent results extend theirs for the mod $p$ Buchstaber invariants, but we prove them using another method.

Let $\Delta^m$ be the $m$-simplex on the set of vertices $[m+1]:=\{1,\ldots,m+1\}$ and let $\Delta_{(k)}^m$ be its $k$-skeleton.

\begin{prop}Let $p$ be a prime number and $1\le k\le m$. A non-degenerate simplicial map $f \colon \Delta_{(k)}^{m} \longrightarrow \mathcal{X}^n$ is an embedding.
\end{prop}
\begin{proof}
Every two vertices in $\Delta_{(k)}^{m}$
are connected by an edge, so the images of vertices by a non-degenerate map must be distinct. Therefore, $f$ is an embedding.
\end{proof}

\begin{prop}\label{emb}
For every prime number $p$ and every $k\le m$ we have
\[
s_p(\Delta_{(k+1)}^{m+1})\leq s_p(\Delta_{(k)}^{m})\leq s_p(\Delta_{(k)}^{m+1}) \leq s_p(\Delta_{(k)}^{m}) +1.
\]
\end{prop}

\begin{proof}
Let  $f \colon \Delta_{(k+1)}^{m+1} \rightarrow \mathcal{X}^n$ be an embedding. Without losing on
generality we may assume that $f(m+2)=e_n$.
Let us consider a map $\bar{f} \colon \Delta_{(k)}^m \rightarrow \mathcal{X}^{n-1}$ defined by $\bar{f}(i) = \pi (f(i))$ where $\pi \colon \F_p^n\rightarrow \F_p^{n-1}$ is the projection on the first $n-1$ coordinates. We will prove that $\bar{f}$ is non-degenerate.

Let us assume the contrary. Then there exist distinct $1 \leq i_1< i_2< \dots< i_j \leq  m + 1$ such that
$j \leq k+1$ and such that $a_1\bar{f}(i_1)+\cdots+ a_j\bar{f}(i_j) = (0, \dots, 0)\in \F_p^{n-1}$ for some $(a_1,\ldots,a_j)\in\F_p^j\setminus\{(0,\ldots,0)\}$. We know that $b_1f(i_1)+\cdots+b_jf(i_j) \neq (0, \dots, 0)\in \F_p^{n}$ for all $(b_1,\ldots,b_j)\in\F_p^j\setminus\{(0,\ldots,0)\}$, so $\pi (a_1f(i_1)+\cdots+a_jf(i_j))= a_1\bar{f}(i_1)+\cdots+ a_j\bar{f}(i_j) = (0, \dots, 0)$ implies that $a_1f(i_1)+\cdots+a_jf(i_j)=r e_n$ for some $1\leq r \leq p-1$. Since $j + 1 \leq k + 2$ the set $\{f(i_1), \ldots, f(i_j), f(m + 2)=e_n\}$ is not a simplex in $\mathcal{X}^n$ contradicting a nondegeneracy of $f$.
Thus we constructed a nondegenerate map from $\Delta_{(k)}^m$ to $\mathcal{X}^{n-1}$. The minimal $r$ such that there exists nondegenerate map $\Delta_{(k)}^m\to \mathcal{X}^r$ is smaller than $n$, so $s_p(\Delta_{(k)}^m)\ge m+1-(n-1)=s_p(\Delta_{(k+1)}^{m+1})$.

Let $f\colon \Delta_{(k)}^{m}\longrightarrow \mathcal{X}^n$ be a nondegenerate simplicial map. We consider $\mathcal{X}^n$ as a subcomplex of $\mathcal{X}^{n+1}$ by adding $0$ as the last coordinate to all vertices. We also consider $\Delta_{(k)}^m$ as a subcomplex of $\Delta_{(k)}^{m+1}$ by adding vertex $v_{m+1}$. Now we define an extension $f'$ of $f$ by defining $f'(v_{m+1})= e_{n+1}$. The map $f'\colon \Delta_{(k)}^{m+1} \longrightarrow \mathcal{X}^{n+1}$ is a nondegenerate simplicial map, and the existence of such map gives the second inequality.

The third inequality is immediate as $\Delta_{(k)}^{m}$ embeds into  $\Delta_{(k)}^{m+1}$ so the minimal $r$ such that there is a non-degenerate map $f \colon \Delta_{(k)}^{m+1} \rightarrow \mathcal{X}^r$ is not smaller than the minimal $n$ such that there is a non-degenerate map $f \colon \Delta_{(k)}^{m+1} \rightarrow \mathcal{X}^n$.
\end{proof}

\begin{prop}
\label{k-sk}
Let $n,m,k \in \N$ be such that $k\leq m$ and
\[
(p-1)\cdot \binom{m}{1}+\cdots +(p-1)^k\cdot \binom{m}{k}<p^n-1.
\]
Then, there exists a nondegenerate simplicial map
\[
f \colon \Delta^m_{(k)} \longrightarrow \mathcal{X}^n.
\]
\end{prop}
\begin{proof}
We fix $k$ and prove the lemma by induction on $m$, for $m\geq k$.

If $m=k$ then given inequality induces $m\leq n-1$ and since $\dim \mathcal{X}^n = n-1$ there is a nondegenerate simplicial map from $\Delta^m_{(k)} \equiv \Delta^m$ into $\mathcal{X}^n$.

Assume as an induction hypothesis that the lemma is true for $m-1$. As usual, vertices of $\Delta^m$ are $[m+1]$ and consider $\Delta^{m-1}$ as a subcomplex of $\Delta^m$ on $[m]$. If the given inequality is true for $m$, it also holds for $m-1$. Therefore, using induction hypothesis, there exists a nondegenerate simplicial map $f\colon \Delta^{m-1}_{(k)}\longrightarrow \mathcal{X}^n$ and we want to extend it over $\Delta^m_{(k)}$ to a simplicial map $f'$. More precisely, we need to show that there exists $x\in \mathcal{X}^n$ so that $\{f(i_0),\ldots ,f(i_s),x\}\in \mathcal{X}^n$, for $1\leq s\leq k-1$, and all $i_0,\ldots ,i_s $. Then we define $f'(m+1):=x$. Using Lemma~\ref{p nms2} (a) we conclude that every element from $\F_p^n\setminus \{0\}$ not belonging to the set
\[
A:= \bigcup_{s=0}^{k-1} \{ a_0f(i_0)+\cdots +a_sf(i_s)\,|\, i_0,\ldots i_s \in \{0,1,\ldots ,m-1\}, a_o,\ldots ,a_s\in \F_p\setminus \{0\}\}
\]
will satisfy this condition.
We have that
\[
|A|\leq (p-1)\binom{m}{1}+\cdots +(p-1)^k\binom{m}{k} < p^n-1,
\]
so the set $(\F_p^n\setminus \{0\})\setminus A$ is nonempty and we pick the value for $f'(m+1)$ from it.
\end{proof}

\begin{cor}
Let $0\leq k\leq m$.
Then
\[
s_p(\Delta^m_{(k)})\geq m+1-\left\lceil \log_p \left(1+\binom{m}{0}+(p-1)\cdot \binom{m}{1}+\cdots +(p-1)^k\cdot \binom{m}{k}\right)\right\rceil.
\]
\end{cor}
\begin{proof}
For $n=\lceil \log_p \big(1+\binom{m}{0}+(p-1)\cdot \binom{m}{1}+\cdots +(p-1)^k\cdot \binom{m}{k}\big)\rceil$ we have from Proposition~\ref{k-sk} that there exists a nondegenerate simplicial map $f \colon \Delta^m_{(k)} \longrightarrow \mathcal{X}^n$. By the definition of Buchstaber invariant, we obtain the desired inequality.
\end{proof}

The obtained estimate is not the best possible; that is, equality may not be true. For example, in the case of $\Delta^6_{(2)}$ we have (\cite[Lemma 3.4]{FukMas} and \cite[Table 2]{10.2969/jmsj/06841695})
that
\[
s_2(\Delta^6_{(2)}) = 3,
\]
while the above inequality gives
\[
s_2(\Delta^6_{(2)})\geq 6+1-\lceil \log_2(23)\rceil = 2.
\]

\begin{prop}
\label{m-k}
Let $0\leq k \leq m-1$ and $p$ a prime.
\begin{itemize}
\item[(a)] Then
\[
1\leq s_p(\Delta^m_{(k)})\leq m-k.
\]
\item[(b)]
If $m\leq p$, then
\[
s_p(\Delta^m_{(k)})=m-k.
\]
\end{itemize}
\end{prop}
\begin{proof}
Let $e_i=(0,\ldots,0,1,0,\ldots,0)$ be the $i$-th standard basis vector of $\F_p^n$.
For $k\leq m-1$ there is a nondegenerate map $f\colon \Delta^m_{(k)} \longrightarrow \mathcal{X}^m$ given by $f(i)=e_i$, for $i=1,\ldots ,m$ and $f(m+1)=(1, 1, \ldots ,1)$. Therefore, the minimal $r$ such that there is a nondegenerate map $f\colon \Delta^m_{(k)} \longrightarrow \mathcal{X}^r$ is smaller or equal to $m$. We get $s_p(\Delta^m_{(k)})=m+1-r \geq m+1-m=1$.

Note that there are no nondegenerate map $f\colon \Delta^m_{(k)} \longrightarrow \mathcal{X}^k$ because
    \[
    \dim (\Delta^m_{(k)})=k>\dim (\mathcal{X}^k)=k-1.
    \]
Therefore, the minimal $r$ such that there is a nondegenerate map $f\colon \Delta^m_{(k)} \longrightarrow \mathcal{X}^r$ is bigger than $k$. So we get $s_p(\Delta^m_{(k)})=m+1-r \leq m+1-(k+1)=m-k$.
That proves part (a).

For part (b), we need to show that the smallest $r$ such that there is a nondegenerate simplicial map $f\colon \Delta^m_{(k)} \longrightarrow \mathcal{X}^r$ is $r=k+1$. We already proved in (a) that $r\geq k+1$.
Therefore, we need to produce a nondegenerate map $f\colon \Delta^m_{(k)} \longrightarrow \mathcal{X}^{k+1}$. Define the map $f$ on vertices of $\Delta^m_{(k)} $ by
    \[
    f(i)=(1,i-1,(i-1)^2,\ldots , (i-1)^k), i=1,2,\ldots ,m,\text{ and } \, f(m+1)=(0,\ldots, 0,1)=e_{k+1}.
    \]

Take a maximal simplex $(i_1,\ldots ,i_{k+1})$, $1\leq i_1<i_2<\ldots <i_{k+1}\leq m$ from $\Delta^m_{(k)}$. Then $(f(i_1), \ldots ,f(i_{k+1}))$ is a simplex in $\mathcal{X}^{k+1}$ because the matrix that they form is a Vandermonde square matrix and its determinant is equal to
\[
\prod_{1\leq j<l\leq k+1}(i_l-i_j) \neq 0 \pmod p,
\]
because $1\leq |i_l-i_j| \leq p-1$.

If the maximal simplex that we observe contains the vertex $m+1$, then we calculate the appropriate determinant by first developing it using the last row, and we are again left with the determinant of the Vandermonde matrix, which is nonzero in $\F_p$.
\end{proof}

\begin{rem}
\label{0 or 1}
If we want to construct a nondegenerate map $f\colon \Delta^m_{(k)} \longrightarrow \mathcal{X}^r$ we may assume that $f(i)=e_i$ for $i=1,\ldots ,r$.

Suppose this is not the case, and we have a nondegenerate map  $f\colon \Delta^m_{(k)} \longrightarrow \mathcal{X}^r$. First, suppose that $r$ is the smallest so that such a nondegenerate map $f$ exists. Then $\Lin(\im(f))=\F_p^r$.

We may rearrange the order of vertices of $\Delta^m$ so that $\{f(1),\ldots,f(r)\}$ is a basis of $\F_p^r$. There is a nondegenerate map $f_A\colon \mathcal{X}^r \longrightarrow \mathcal{X}^r$ induced by a bijective linear map $A\colon \F_p^{r}\to \F_p^{r}$ such that $A(f(i))=e_i$ for $i=1,\ldots, r$ and the composition $f_A\circ f$ is a desired nondegenerate map.

If $r$ is not the smallest so that such a nondegenerate map $f$ exists, let $s=\dim \Lin(\im(f))$. Then, using the same reasoning as above, we may assume $f(i)=e_i$, $i=1,\ldots ,s$. Moreover $\im(f)\subset \F_p^s\subset \F_p^r$. We may redefine $f(i)=e_i$, $i=s+1, \ldots, r$ and $f$ will remain nondegenerate.

Furthermore, we may assume that all coordinates of $f(r+1)$ are either 0 or 1. If not, let $i_1,\ldots,i_t$ be nonzero coordinates of $f(r+1)$. Denote by $f(i)_j$ the $j$-th coordinate of $f(i)$. Then
\[
a_1e_1+\ldots +a_{r}e_{r}+a_{r+1}f(r+1)+a_{r+2}f(r+2)+\ldots+a_{m+1}f(m+1)=0
\]
if and only if
\[
a_1\delta_1e_1+\ldots +a_{r}\delta_{r}e_{r}+a_{r+1}f^*(r+1)+a_{r+2}f^*(r+2)+\ldots+a_{m+1}f^*(m+1)=0,
\]
where $\delta_j=f(r+1)_j^{-1}$, for $j=i_l$, for $l=1,\ldots ,t$, and $\delta_j=1$, otherwise, and $f^*(l)_j=\delta_j f(l)_j$ for $l=r+1,\ldots,m+1$ and all $j=1,\ldots,r$. Practically, the $j$-th coordinate of the second linear combination is obtained from the $j$-th coordinate of the first one multiplied by $\delta_j$. Therefore $f^*$ is a nondegenerate map, $f^*(i)=e_i$, for $i=1,\ldots ,r$, and $f^*(r+1)$ has all coordinates equal to 0 or 1.
\end{rem}

\begin{prop}
\label{sp>1}
   Let $m\geq 2$ and $0\leq k\leq m$.
  \begin{itemize}
\item[(a)]
 If $k\neq p-1 \pmod p$, then
 $s_p(\Delta^m_{(k)})\geq 2$ if and only if $m\geq k + \left[\frac{k}{p}\right]+2$.
 \item[(b)]
 If $k= p-1 \pmod p$, then
 $s_p(\Delta^m_{(k)})\geq 2$ if and only if $m\geq k + \left[\frac{k}{p}\right]+3$.
  \end{itemize}
\end{prop}
\begin{proof}
    We have that $s_p(\Delta^m_{(k)})\geq 2$ if and only if there is a nondegenerate map $f \colon \Delta^m_{(k)} \longrightarrow \mathcal{X}^{m-1}$. If we try to construct such a map $f$, we may assume that $f(i)=e_i$, for $i=1,\ldots ,m-1$ and that all coordinates of $f(m)$ are either 0 or 1 (see Remark~\ref{0 or 1}). It remains to define $f(m)$ and $f(m+1)$ so that $f$ is a nondegenerate map.

    Denote by $x_1$ the number of nonzero coordinates of $f(m)$, which are zero in $f(m+1)$, by $x_2$ the number of nonzero coordinates of $f(m+1)$, which are zero in $f(m)$, and by $x_{11}$ the number of coordinates which are nonzero in both $f(m)$ and $f(m+1)$.
   In order that $\{f(i_1),\ldots ,f(i_k),f(m)\}$ form a simplex, for all $1\leq i_1 < \ldots <i_k\leq m-1$ we need to have
    \begin{equation}
    \label{sp=1 1}
    x_1+x_{11}\geq k+1.
     \end{equation}
    Similarly, for $f(m+1)$, we need to have
    \begin{equation}
    \label{sp=1 2}
    x_2+x_{11}\geq k+1.
      \end{equation}

Finally, in order that $\{f(i_1),\ldots ,f(i_{k-1}),f(m),f(m+1)\}$ form a simplex, for all $1\leq i_1 < \ldots <i_{k-1}\leq m-1$ we need to have
\begin{equation}
 \label{sp=1 3a}
x_1+x_2+ x_{11}-M \geq k
\end{equation}
where $M$ is the maximal number of equal coordinates of $f(m+1)$ which are nonzero in $f(m)$. The number on the left hand side of the previous inequality is the minimal number of nonzero coordinates among all linear combinations of the form $af(m)+bf(m+1)$, $a,b \neq 0$.

We need $M$ to be as small as possible, and it will happen if we distribute all coordinates $1,2,\ldots ,p-1$ evenly at $x_{11}$ places. More precisely, $M= [\frac{x_{11}+p-2}{p-1}]$ and our equation~\eqref{sp=1 3a} became
\begin{equation}
\label{sp=1 3}
    x_1+x_2+x_{11}-\left[\frac{x_{11}+p-2}{p-1}\right]\geq k.
\end{equation}

To summarize, we need to have a nonnegative integer solution of the system of inequalities
\begin{equation}
\label{system sp=1}
    x_1+x_{11}\geq k+1, \,\,
    x_2+x_{11}\geq k+1,\,\,
    x_1+x_2+x_{11}-\left[\frac{x_{11}+p-2}{p-1}\right]\geq k.
\end{equation}
which satisfies
 the necessary condition
\begin{equation}
\label{sp=1 cond}
    x_1+x_2+x_{11} \leq m-1.
\end{equation}

 For the proof of part (a), we put equalities instead of inequalities in the system~\eqref{system sp=1} to obtain the system of equations with a unique solution. Denote $k=\alpha \cdot p + \beta$, for $\alpha \geq 0,\, 0\leq \beta \leq p-2$, and $x_{11}=A\cdot (p-1)+1+B$, for $0\leq B\leq p-2$. We get $\left[\frac{x_{11}+p-2}{p-1}\right] = A+1$ and the unique solution of the system is
 \[
 x_1^0=x_2^0=\alpha , \,\, x_{11}^0=\alpha\cdot (p-1)+1+\beta=k+1-\alpha.
 \]
We have
\[
x_1^0+x_2^0+x_{11}^0=\alpha\cdot (p+1)+1+\beta = k+\alpha +1= k+\left[\frac{k}{p}\right]+1.
\]

Let $x_1, x_2, x_{11}$ be a solution of the system~\eqref{system sp=1}. Let $\epsilon_I\geq 0, I\in \{1,2,11\}$ denote by how much appropriate inequality differs from equality. For example, $\epsilon_1 := x_1+x_{11}-k-1$. Denote again $x_{11}=A\cdot (p-1)+1+B$, for $0\leq B\leq p-2$.
We have $\left[\frac{x_{11}+p-2}{p-1}\right] = A+1$ and~\eqref{sp=1 3} gives
\[
k+1+\epsilon_1+k+1+\epsilon_2-A(p-1)-1-B-A-1=k+\epsilon_{11}
\]
\[
Ap+B=k+\epsilon_1+\epsilon_2-\epsilon_{11}
\]
Put $k=\alpha \cdot p + \beta$, for $\alpha \geq 0,\, 0\leq \beta \leq p-2$ and we get
\[
(A+\epsilon_{11}-\alpha )p = \epsilon_1+\epsilon_2+(p-1)\epsilon_{11} + (\beta -B)\geq 0+0+0+0-(p-2)=-(p-2).
\]
Since all values involved are integers, we conclude that
\[
A+\epsilon_{11}-\alpha\geq 0.
\]
From that inequality and~\eqref{sp=1 3} we get
\[
x_1+x_2+x_{11}=k+\epsilon_{11}+A+1\geq k+\alpha +1=x_1^0+x_2^0+x_{11}^0.
\]
Therefore, the minimal value of $x_1+x_2+x_{11}$ is $x_1^0+x_2^0+x_{11}^0$ and we need it to be smaller or equal than $m-1$. We get
\[
x_1^0+x_2^0+x_{11}^0= k+\left[\frac{k}{p}\right]+1\leq m-1
\]
which completes the proof.

For the proof of part (b), put $k=\alpha \cdot p +p-1$. Let $x_1,x_2, x_{11}$ be a solution of the system~\eqref{system sp=1}, with $\epsilon_I\geq 0$ defined as in part (a) and put $x_{11}=A\cdot (p-1)+1+B$, $0\leq B\leq p-2$. We calculate $x_1$ and $x_2$ from~\eqref{sp=1 1} and~\eqref{sp=1 2} and put in~\eqref{sp=1 3} to get
\[
k+1+\epsilon_1-x_{11}+k+1+\epsilon_2-x_{11} +x_{11}-A-1=k+\epsilon_{11}
\]
When we rearrange summands, we get
\[
(A+\epsilon_{11} -\alpha -1)\cdot p = \epsilon_1+\epsilon_2 +(p-1)\epsilon_{11}-B-1\geq -p+1.
\]
Since we are working with integers we conclude that $A+\epsilon_{11} -\alpha -1\geq 0$.
From~\eqref{sp=1 3} we get
\[
x_1+x_2+x_{11} = k+A+1+\epsilon_{11}\geq k+1+\alpha+1= k+\left[\frac{k}{p}\right]+2.
\]
For $\epsilon_1=\epsilon_2=0$ and $\epsilon_{11}=1$ we have a solution of the system~\eqref{system sp=1} which achieves that minimum:
\[
x_1=x_2=\left[\frac{k}{p}\right]+1, \,\,x_{11}=k-\left[\frac{k}{p}\right].
\]
Therefore, there is a solution of the system~\eqref{system sp=1} which satisfies~\eqref{sp=1 cond} if and only if
\[
k+\left[\frac{k}{p}\right]+2\leq m-1.\qedhere
\]
\end{proof}

Note that Proposition~\ref{m-k}(a) and Proposition~\ref{sp>1} gives us all values of the pairs $(m,k)$ for which $s_p(\Delta^m_{(k)})=1$. We may rewrite the conditions from Proposition~\ref{sp>1} to get the following.
\begin{cor}
\label{sp=1 cor}
    Let $m\geq 2, \, 0\leq k\leq m$ and $p$ a prime. Then $s_p(\Delta_{(k)}^m)=1$ if and only if
    \[
    m-\left[\frac{m}{p+1}\right]-1\leq k \leq m-1,\text{ for } k\neq p-1 \pmod p
    \]
    or
    \[
   p\cdot \left[\frac{m-1}{p+1}\right]+p-1 \leq k \leq m-1, \text{ for } k=p-1 \pmod p.
    \]
\end{cor}
\begin{proof}
For $k\neq p-1 \pmod p$ we have from Proposition~\ref{m-k} and Proposition~\ref{sp>1} that
$s_p(\Delta_{(k)}^m)=1$ if and only if
\[
k + \left[\frac{k}{p}\right]> m-2\text{ and } k\leq m-1.
\]
Denote $k=\alpha \cdot p + \beta$, with $0\leq \beta \leq p-2$ and the first inequality becomes
\[
\alpha \cdot p + \beta +\alpha \geq m-1
\]
\[
\alpha \cdot (p+1)\geq m-1-\beta
\]
  If $\alpha \geq \left[\frac{m}{p+1}\right]+1$ the inequality holds, while for $\alpha \leq \left[\frac{m}{p+1}\right]-1$ it is false. For $\alpha = \left[\frac{m}{p+1}\right]$ we have
  \[
  \beta \geq m-1 - \alpha (p+1) = m-1 - \left[\frac{m}{p+1}\right]\cdot (p+1)
  \]
  which gives us
  \[
  k=\alpha p+\beta \geq \left [\frac{m}{p+1}\right ]\cdot p + m-1 - \left [\frac{m}{p+1}\right ]\cdot (p+1) = m-\left [\frac{m}{p+1}\right]-1.
  \]

For $k= p-1 \pmod p$ we have from Proposition~\ref{m-k} and Proposition~\ref{sp>1} that
$s_p(\Delta_{(k)}^m)=1$ if and only if
\[
k + \left[\frac{k}{p}\right]> m-3\text{ and } k\leq m-1.
\]
Denote $k=\alpha \cdot p + p-1$ and the first inequality becomes
\[
\alpha \cdot p + p-1 +\alpha \geq m-2
\]
\[
\alpha +1 \geq \frac{m}{p+1}\text{ \,  and this is equivalent to  \, } \alpha \geq  \left [\frac{m-1}{p+1}\right ]
\]
Therefore,
\[
k=\alpha \cdot p + p-1 \geq  p\cdot \left [\frac{m-1}{p+1}\right]+p-1.
\]

\end{proof}
As an easy application of the Corollary~\ref{sp=1 cor} one may get the following.

\begin{cor} We have for every prime $p$:
\begin{itemize}
\item[(a)] $s_p(\Delta^m_{(m-1)})=1$, for all $m\geq 2$.
\item[(b)] $s_p(\Delta^m_{(m-2)})=1$ if and only if $m\geq p+1$.
    \end{itemize}
\end{cor}

Table \ref{TabelBuchstaber3} gives the mod 3 Buchstaber invariant values of the skeletons of the simplex of dimension not greater than ten. The red values indicate that the mod 3 is distinct from the corresponding real (the mod 2) Buchstaber invariant.

\begin{center}
\begin{table}[ht]\label{s3tab}
\begin{tabular}{|c|c c c c c c c c c c c c c c c|}
\hline
$m$\textbackslash$k$ & 0 & 1 & 2 & 3 & 4 & 5 & 6 & 7 & 8 & 9 & 10 & 11 & 12 & 13 & 14\\
\hline
2 & 2 & 1 & 0 &  &  &  & & &  & & & & & &\\
3 & 3 & 2 & 1 & 0 &  &  & & &  & & & & & &\\
4 & 4 &	2 &	1 &	1 &	0 &	& & &  & & & & & &\\
5 &	5 &	3 &	2 &	1 &	1 &	0 &	 &	 &  & & & & & &\\
6 &	6 &	4 &	3 &	\textcolor{red}{2} &	1 &	1 &	0 &	 &  & & & & & &\\
7 &	7 &	5 &	4 &	\textcolor{red}{3} &	2 &	1 &	1 &	0 &  & & & & & &\\
8 &	8 &	6 &	\textcolor{red}{5} &	\textcolor{red}{4} & \textcolor{red}{3}	 &	\textcolor{red}{1} &	1 &	1 & 0 & & & & & &\\
9 &	9 &	7 &	\textcolor{red}{6} & \textcolor{red}{5}	 & \textcolor{red}{4}	 & \textcolor{red}{2}	 & \textcolor{red}{1}	 &	1 & 1 & 0 & & & & &\\
10 &10&	8 &	 &	 &	 &	 &	 &	 & 1 & 1 &0 & & & &\\
\hline
\end{tabular}
\\[1mm]
\caption{The table of the mod $3$ Buchstaber invariants for $\Delta_{(k)}^m$. }\label{TabelBuchstaber3}
\end{table}
\end{center}

In the next table we compare the mod $p$ Buchstaber invariants of the skeletons of simplex $\Delta^m$ for $m\leq 9$ and $p\in\{2, 3, 5, 7\}$. The values presented in Table \ref{TabelBuchstaberp} are obtained either from previously established claims and results from \cite{Baralic2023}, or through direct computation.

\begin{center}
\begin{table}[ht]\label{sptab}
\begin{tabular}{|c|c c c c c c c c c|}
\hline
$m$\textbackslash$k$ & 1 & 2 & 3 & 4 & 5 & 6 & 7 & 8 & 9\\
\hline
2 &  1,\textcolor{red}{1},\textcolor{blue}{1},\textcolor{green}{1} & 0 &  &  &  & & &  &\\
3 & 1,\textcolor{red}{2},\textcolor{blue}{2},\textcolor{green}{2} & 1,\textcolor{red}{1},\textcolor{blue}{1},\textcolor{green}{1} & 0 &  &  & & &  &\\
4 & 2,\textcolor{red}{2},\textcolor{blue}{3},\textcolor{green}{3} & 1,\textcolor{red}{1},\textcolor{blue}{2},\textcolor{green}{2} & 1,\textcolor{red}{1},\textcolor{blue}{1},\textcolor{green}{1} &	0 &	& & &  &\\
5 & 3,\textcolor{red}{3},\textcolor{blue}{4},\textcolor{green}{4} & 2,\textcolor{red}{2},\textcolor{blue}{3},\textcolor{green}{3} & 1,\textcolor{red}{1},\textcolor{blue}{2},\textcolor{green}{2} & 1,\textcolor{red}{1},\textcolor{blue}{1},\textcolor{green}{1} &	0 &	 &	 &  &\\
6 & 4,\textcolor{red}{4},\textcolor{blue}{5},\textcolor{green}{5} & 3,\textcolor{red}{3},\textcolor{blue}{4},\textcolor{green}{4} & 1,\textcolor{red}{2},\textcolor{blue}{2},\textcolor{green}{3} & 1,\textcolor{red}{1},\textcolor{blue}{1},\textcolor{green}{2} & 1,\textcolor{red}{1},\textcolor{blue}{1},\textcolor{green}{1} &	0 &	 &  &\\
7 & 4,\textcolor{red}{5},\textcolor{blue}{5},\textcolor{green}{6} & 4,\textcolor{red}{4},\textcolor{blue}{4},\textcolor{green}{5} & 2,\textcolor{red}{3},\textcolor{blue}{3},\textcolor{green}{4} & 1,\textcolor{red}{2},\textcolor{blue}{2},\textcolor{green}{3} & 1,\textcolor{red}{1},\textcolor{blue}{1},\textcolor{green}{2} & 1,\textcolor{red}{1},\textcolor{blue}{1},\textcolor{green}{1} &	0 &  &\\
8 & 5,\textcolor{red}{6},\textcolor{blue}{6},\textcolor{green}{6}& 4,\textcolor{red}{5},\textcolor{blue}{5}\phantom{,4} & 2,\textcolor{red}{4},\textcolor{blue}{4}\phantom{,4} & 2,\textcolor{red}{3},\textcolor{blue}{3}\phantom{,4} & 1,\textcolor{red}{1},\textcolor{blue}{2}\phantom{,4} & 1,\textcolor{red}{1},\textcolor{blue}{1},\textcolor{green}{1} & 1,\textcolor{red}{1},\textcolor{blue}{1},\textcolor{green}{1}& 0 &\\
9 & 6,\textcolor{red}{7},\textcolor{blue}{7},\textcolor{green}{7} & 5,\textcolor{red}{6}\phantom{,4,4} & 3,\textcolor{red}{5}\phantom{,4,4} & 2,\textcolor{red}{4}\phantom{,4,4} & 1,\textcolor{red}{2}\phantom{,4,4} & 1,\textcolor{red}{1}\phantom{,4,4} & 1,\textcolor{red}{1},\textcolor{blue}{1},\textcolor{green}{1} & 1,\textcolor{red}{1},\textcolor{blue}{1},\textcolor{green}{1}& 0\\

\hline
\end{tabular}
\\[1mm]
\caption{The table of the mod $p$ Buchstaber invariants for $\Delta_{(k)}^m$; black for $p=2$, red for $p=3$, blue for $p=5$ and green for $p=7$. }\label{TabelBuchstaberp}
\end{table}
\end{center}





\section{Mod $p$ Buchtaber invariant of universal complexes $X(\F_p^n)$}

The first impression is that $s_p(\mathcal{K})<s_q(\mathcal{K})$, for $p<q$. For such a claim, it is necessary and sufficient that there exists a nondegenerate map $X(\F_p^n) \longrightarrow X(\F_q^n)$, for $p<q$. But this is not the case, as the next theorem shows.

\begin{thm}
\label{nonexistence}
  There are no nondegenerate map $f\colon X(\F_2^4)\to X(\F_3^4)$.
\end{thm}
\begin{proof}
Suppose there exists a nondegenerate map $f\colon X(\F_2^4)\to X(\F_3^4)$. We may assume (\Cref{0 or 1}) that $f(v)=v$ for $v=e_1,e_2,e_3,e_4,e=(1,1,1,1)$. Regarding $f(e)$, if the $i$-th coordinate of $f(e)$ is 0 for some $i$, then $\{e,e_a,e_b,e_c\}$ is a 3 simplex, for $\{i,a,b,c\}=[4]$, but $f(\{e,e_a,e_b,e_c\})$ is not. This is not possible since $f$ is nondegenerate. Therefore, all coordinates of $f(e)$ are nonzero, and we apply Remark~\ref{0 or 1} for $r+1$.

For $v\in\F_p^4$ we denote by $v_i$ its $i$-th coordinate and define $I(v)=\{i\,|\,v_i\ne 0\}$.

\begin{lemma}\label{Propertiesf}
Let $v\in \F_2^4$.
\begin{itemize}
    \item[(a)] We have $I(v)\subset I(f(v))$.
    \item[(b)] If $I(v)\ne I(f(v))$ then $f(v)_i=1$ if and only if $i\in I(v)$ or $f(v)_i=2$ if and only if $i\in I(v)$.
    \item[(c)] If there are $i,j\in I(v)$ such that $f(v)_i\ne f(v)_j$ then $I(v)= I(f(v))$.
\end{itemize}
\end{lemma}

\begin{proof}
\begin{itemize}
    \item[(a)] Let $i\in I(v)$. Then $\{v,e_a,e_b,e_c\}$ is a 3-simplex in $X(\F_2^4)$, where$\{i,a,b,c\}=[4]$. If $f(v)_i=0$ then $f(\{v,e_a,e_b,e_c\})$, is not a 3-simplex in $X(\F_3^4)$. It is impossible since $f$ is a nondegenerate map, so $i\in I(f(v))$.
    \item[(b)] Suppose there exists $i\in I(f(v))\setminus I(v)$. If there is $j\in I(v)$ such that $f(v)_i=f(v)_j$, then $f(\{v,e,e_a,e_b\})$, where $\{i,j,a,b\}=[4]$, is not a 3-simplex. Therefore $f(v)_i\ne f(v)_j$ for all $j\in I(v)$. In particular $f(v)_j=f(v)_k\notin \{0,f(v)_i\}$, for all $j,k\in I(v)$. Moreover, $f(v)_j\in \{0, f(v)_i\}$, for all $j \notin I(v)$.
    \item[(c)] It follows from (2). \qedhere
\end{itemize}
\end{proof}

Let $v=(0,0,1,1)$. By the above lemma $f(v)_3\ne 0$, so by \Cref{nondeg} from now on we assume that $f(v)_3=1$. Also by the lemma we have that $I(f(v))=I(v)=\{3,4\}$ or $f(v)$ equals $(2,0,1,1)$, $(0,2,1,1)$ or $(2,2,1,1)$.
In the following four lemmas, we show that neither case is valid, thus proving the theorem.
\end{proof}

\begin{lemma}\label{Propertyf11}
The last two coordinates of $f(v)$ are equal and nonzero.
\end{lemma}

\begin{proof}
By the previous lemma, we have that the last two coordinates of $f(v)$ are nonzero. Suppose that they are not equal.
By \Cref{nondeg} we may assume that $f(v)=(0,0,1,2)$.

First we prove that for $u=(x,y,0,1)$ and $u=(x,y,1,0)$ we have $f(u)=u$ or $f(u)=2u$. We consider the case $u=(x,y,0,1)$, and the proof for the other case is the same. Suppose $f(u)\notin \{u,2u\}$. By the assumption, at least one of $x$ and $y$ is nonzero. We may assume that $y=1$, that is $u=(x,1,0,1)$.

If $f(u)_3\ne0$, by \Cref{Propertiesf} we have $f(u)=(x',1,2,1)$ or $(x',2,1,2)$. In both cases $f(\{v,u,e_1,e_2\})$ is not a 3-simplex, while $\{v,u,e_1,e_2\}$ is a 3-simplex, which is a contradiction. So $f(u)_3=0$.

If $x=0$, then by \Cref{Propertiesf} we have $f(u)\in \{(2,1,0,1),(1,2,0,2), (0,1,0,1), (0,2,0,2)\}$. In first two cases $f(\{v,u,e,e_2\})$ is not a 3-simplex, while $\{v,u,e,e_2\}$ is. So $f(u)=u$ or $f(u)=2u$.

For $x=1$ (and $f(u)_3=0$), by \Cref{Propertiesf}, there are eight possibilities for $f(u)$. If $f(u)$ is $(1,1,0,2)$, $(2,2,0,1)$, $(1,2,0,1)$ or $(2,1,0,2)$ then $f(\{u,v,e,e_1\})$ is not a 3-simplex, while $\{u,v,e,e_1\}$ is. If $f(u)=(2,1,0,1)$ or $(1,2,0,2)$ then $f(\{u,v,e,e_2\})$ is not a 3-simplex, while $\{u,v,e,e_2\}$ is. Hence $f(u)=u$ or $f(u)=2u$. So we may assume (Proposition~\ref{nondeg}) that $f(u)=u$ for $u=(1,0,1,0)$, $(0,1,1,0)$, $(1,1,1,0)$, $(1,0,0,1)$, $(0,1,0,1)$, and $(1,1,0,1)$.

Now we consider $u=(1,1,0,0)$. From Proposition~\ref{nondeg} and Lemma~\ref{Propertiesf} we get that there are five essentially different possibilities for $f(u)$: $(1,1,0,0)$, $(1,2,0,0)$, $(1,1,2,0)$, $(1,1,0,2)$, $(1,1,2,2)$, and for $w=(0,1,1,1)$ we also five essentially different possibilities for $f(w)$: $(0,1,1,1)$, $(0,1,1,2)$, $(0,1,2,1)$, $(0,2,1,1)$, $(1,2,2,2)$.

Suppose $f(u)=(1,1,0,0)$. If $f(w)=(0,1,1,1)$, then $f(\{w,u,v,e_1+e_4\})$ is not a 3-simplex. If $f(w)=(0,1,1,2)$, then $f(\{w,u,e,e_2+e_4\})$ is not a 3-simplex. If $f(w)=(0,1,2,1)$, then $f(\{w,u,e,e_2+e_3\})$ is not a 3-simplex. If $f(w)=(0,2,1,1)$, then $f(\{w,u,v,e_2+e_4\})$ is not a 3-simplex. If $f(w)=(1,2,2,2)$, then $f(\{w,u,v,e_2+e_4\})$ is not a 3-simplex.

Suppose $f(u)=(1,2,0,0)$. If $f(w)=(0,1,1,1)$, then $f(\{w,u,v,e+e_3\})$ is not a 3-simplex. If $f(w)=(0,1,1,2)$, then $f(\{w,u,e,e_3\})$ is not a 3-simplex. If $f(w)=(0,1,2,1)$, then $f(\{w,u,e,e_4\})$ is not a 3-simplex. If $f(w)=(0,2,1,1)$, then $f(\{w,v,e_1,e+e_3\})$ is not a 3-simplex. If $f(w)=(1,2,2,2)$, then $f(\{w,u,v,e_4\})$ is not a 3-simplex.

Suppose $f(u)=(1,1,0,2)$. If $f(w)=(0,1,1,1)$, then $f(\{w,u,e_2,e_1+e_3\})$ is not a 3-simplex. If $f(w)=(0,1,1,2)$, then $f(\{w,u,e_1,e_3\})$ is not a 3-simplex. If $f(w)=(0,1,2,1)$, then $f(\{w,u,e_1,e_2+e_3\})$ is not a 3-simplex. If $f(w)=(0,2,1,1)$, then $f(\{w,u,e_1,e_3\})$ is not a 3-simplex. If $f(w)=(1,2,2,2)$, then $f(\{w,u,e_2,e_3\})$ is not a 3-simplex.

Suppose $f(u)=(1,1,2,0)$. If $f(w)=(0,1,1,1)$, then $f(\{w,u,e_2,e_1+e_4\})$ is not a 3-simplex. If $f(w)=(0,1,1,2)$, then $f(\{w,u,e_1,e_2+e_4\})$ is not a 3-simplex.  If $f(w)=(0,1,2,1)$, then $f(\{w,u,e_1,e_4\})$ is not a 3-simplex. If $f(w)=(0,2,1,1)$, then $f(\{w,u,e_1,e_4\})$ is not a 3-simplex. If $f(w)=(1,2,2,2)$, then $f(\{w,u,e_2,e_4\})$ is not a 3-simplex.

If $f(u)=(1,1,2,2)$, then $f(\{u,v,e_2,e_1+e_4\})$ is not a 3-simplex.
\end{proof}

\begin{lemma}
Let $v=(0,0,1,1)$, then $f(v)\ne (2,0,1,1)$ and $f(v)\ne (0,2,1,1)$.
\end{lemma}

\begin{proof}
By symmetry, it is enough to show that $f(v)\ne (2,0,1,1)$. Assume that $f(v)= (2,0,1,1)$.

Let $u=(1,1,0,0)$.  By \Cref{nondeg}, \Cref{Propertiesf}, and \Cref{Propertyf11}  we may assume that $f(u)\in$ $\{(1,1,2,0)$, $(1,1,0,2)$, $(1,1,2,2)$, $(1,1,0,0)\}$.
If $f(u)=(1,1,2,0)$ or $f(u)=(1,1,2,2)$, then $f(\{u,v,e_2,e_4\})$ is not a 3-simplex. If $f(u)=(1,1,0,2)$, then $f(\{u,v,e_2,e_3\})$ is not a 3-simplex. Hence $f(u)=(1,1,0,0)$.

Let $u=(0,1,0,1)$. Using \Cref{nondeg}, \Cref{Propertiesf}, and \Cref{Propertyf11} we have $f(u) \in$ $\{(0,1,0,1)$, $(2,1,0,1)$, $(0,1,2,1)$, $(2,1,2,1)\}$. If $f(u)=(0,1,0,1)$ then $f(\{u,v,e_3,e_1+e_2\})$ is not a 3-simplex. If $f(u)=(0,1,2,1)$ or $f(u)=(2,1,0,1)$ then $f(\{u,v,e,e_4\})$ is not a 3-simplex. If $f(u)=(2,1,2,1)$ then $f(\{u,v,e_4,e_1+e_2\})$ is not a 3-simplex. Therefore, $f(u)$ cannot be defined so that $f$ is nondegenerate, and we proved that the assumption is wrong.
\end{proof}

\begin{lemma}
Let $v=(0,0,1,1)$, then $f(v)\ne (2,2,1,1)$.
\end{lemma}

\begin{proof}
Assume that $f(v)= (2,2,1,1)$. Let $u=(1,1,0,0)$. By previous lemmas and \Cref{nondeg} we have $f(u)=(1,1,2,2)$. But $f(\{u,v\})$ is not a 1-simplex which is a contradiction.
\end{proof}

\begin{lemma}
Let $v=(0,0,1,1)$, then $f(v)\ne (0,0,1,1)$.
\end{lemma}

\begin{proof}
Assume that $f(v)= (0,0,1,1)$. Let $u=(0,1,1,1)$.  By \Cref{nondeg} and \Cref{Propertiesf} we may assume that $f(u)\in$ $\{(2,1,1,1)$, $(0,1,1,1)$, $(0,2,1,1)$, $(0,1,2,1)$, $(0,2,2,1)\}$. In all possible cases $\{e_1+e_4,e_1+e_3,e_1+e_2,u\}$ is a 3-simplex but
$f(\{e_1+e_4,e_1+e_3,e_1+e_2,u\})$ is not a 3-simplex.
\end{proof}

Now we are able to calculate $s_p(X(\F_2^4))$ for all primes $p$.

\begin{prop} Let $p$ be a prime number. Then
\[  s_p(X(\F_2^4))=11\text{ for } p\neq 3,\]
\[ s_3(X(\F_2^4))=10.\]
\end{prop}
\begin{proof}
From Theorem~\ref{nonexistence} the smallest $r$ so that there is a nondegenerate map from $X(\F_2^4)$ into $X(\F_3^r)$ is $\geq 5$.
Now we define a nondegenerate map from $X(\F_2^4)$ into $X(\F_3^5)$. Define $f \colon X(\F_2^4) \rightarrow  X(\F_3^5)$ by setting $f(v)=v$ if $I(v)\neq 3$ and $f(v)=v+e_5$ if $I(v)= 3$, where $v$ in the image assumes the vector in $X(\F_3^5)$ with the same coordinate and the fifth coordinate equal to zero. Hadamard's matrices having determinant equal to $\pm 3$ come from permutations of four vectors $v$ such that $I(v)=3$  or permutations of the vectors of the following type $e_a+e_b+e_c, e_d+e_a, e_c+e_a, e_b+e_a$ such that $\{a, b, c, d\}=[4]$, see \cite{ZIVKOVIC2006310}.
It is easy to verify that $f$ is nondegenerate. Therefore, \[s_3(X(\F_2^4))=15-5=10.\]

The identity map on $X(\F_2^4)$ gives us that $s_2(X(\F_2^4))=15-4=11.$

For $p\geq 5$, we have a map $f \colon X(\F_2^4) \rightarrow  X(\F_p^4)$ induced by mapping a vector in $\F_2^4$ to the vector with the same coordinates in $\F_p^4$. The absolute determinant of $4\times 4$ matrix with 0 and 1 entries is at most three, so it is less than $p$, and the map $f$ is nondegenerate. Therefore, $s_p(X(\F_2^4))=15-4=11.$
\end{proof}

Although it looks like $s_p(Y)\leq s_q(Y)$, for $p<q$, it is not the case.
We may say that in general, mod $p$ Buchstaber invariants for various primes $p$ are not comparable, as we just proved that
\[
s_2(X(\F_2^4)) > s_3(X(\F_2^4))
\]
and
\[
s_3(X(\F_2^4)) < s_5 (X(\F_2^4)).
\]

Here are a few more examples of calculations of mod $p$ Buchstaber invariant.

We start with a trivial example.
\begin{example}
For every prime $p$ and every $n$ we have
    \[ s_p(X(\F_p^n))= p^n-1-n.\]
\end{example}
\begin{proof}
    The identity map on $X(\F_p^n)$ is a nondegenerate map, so the smallest $r$ equals $n$, by dimensional reason. We get the result since $X(\F_p^n)$ has $p^n-1$ vertices.
\end{proof}

\begin{example} We have
     \[s_2(X(\F_3^3))=21.\]
\end{example}

\begin{proof} First we show that there is a non-degenerate map from $X(\F_3^3)$ into $X(\F_2^5)$. Observe that 26 vertices of $X(\F_3^3)$ can be grouped into thirteen couples $\{v, 2v\}$. The pairing induces a non-degenerate map from $X(\F_3^3)$ onto $\Delta^{12}_{(2)}$ coming from a bijection among the pairs and the vertices of  $\Delta^{12}$. Indeed, it is easy to see that $\Delta^{12}_{(2)}$ can be embedded into $X(\F_2^5)$ by mapping the vertices into $e_1, e_2, \dots, e_5$ and eight vertices of $\F_2^5$ having three nonzero coordinates. Since we have a nondegenerate map from $X(\F_3^3)$ into $\Delta^{12}_{(2)}$,  we obtain the wanted map.

We are going to show that there does not exist a nondegenerate simplicial map $f: X(\F_3^3)\rightarrow X(\F_2^4)$. Assume, on the contrary, that there is a nondegenerate map $f$. If there are two vectors $v_1, v_2\in X(\F_3^3) $ such that $f(v_1)=f(v_2)$   than $v_1=2 v_2$, so $\im f$ has at least thirteen vertices. Since $X(\F_2^4)$ has fifteen, without loss of generality we can assume that $f(e_1)=e_1, f(e_2)=e_2$,  $f(e_3)=e_3$ and $f(e_1+e_2+e_3)=e_1+e_2+e_3$ and that $\im f$ is a surjection or exhibits $e_4$ or $e_1+e_2+e_3+e_4$ or the both. It means that there exist $v_1$, $v_2$ and $v_3\in X(\F_3^3)$ such that $f(v_1)=e_1+e_2$, $f(v_2)=e_2+e_3$ and $f(v_3)=e_3+e_1$. Since $f$ is nondegenerate, $\{e_1, e_2, v_1\}$ is not a simplex of $X (\F_3^3)$ so $v_1=e_1\pm e_2$. Analogously,  $v_2=e_2\pm e_3$ and  $v_3=e_1\pm e_2$. Since $\{e_1+ e_2, e_2+ e_3, e_3+e_1$ is not a simplex of $X(\F_2^4)$  then $\{v_1, v_2, v_3\}$ is not a simplex of $X(\F_3^3)$. Then it means that at least one of $v_1$, $v_2$ and $v_3$ has $-1$, $0$ and $1$ as its coordinates. Assume that $v_1=e_1-e_2$. Then $\{v_1, e_1+e_2+e_3,e_3\}$ is a 2-simplex of $X(\F_3^3)$, but $\{e_1+e_2, e_1+e_2+e_3,e_3\}$ is not a 2-simplex in $X(\F_2^4)$ so $f$ cannot be nondegenerate.
\end{proof}

\begin{example}

We have the following inequality:
\[
    s_2(X(\F_3^4))\leq 73<76 = s_3(X(\F_3^4)).
    \]
\end{example}
\begin{proof}
    In $X(\F_3^4)$, every vertex $v$ is connected with all other vertices except with $2v$, and so it has 78 neighbours. In $X(\F_2^6)$, every vertex is connected with all other vertices, and so it has 62 neighbors.
    Therefore there are no nondegenerate map from $X(\F_3^4)$ into $X(\F_2^6)$ and we may conclude that  $s_2(X(\F_3^4))\leq 80-7=73.$
\end{proof}

\begin{prop} Let $K$ be a simplicial complex on $m$ vertices. Then $s_2 (K)\leq s_p (K)$ for all primes $$p>\frac{(m-s_2 (K)+1)^{\frac{m-s_2 (K)+1}{2}}}{2^{m-s_2 (K)}}.$$
\end{prop}

\begin{proof} Let us denote by $r=m-s_2 (K)$ and let $f\colon K\rightarrow X (\F_2^r)$ be a simplicial nondegenerate map. Let $f_p (v)\in\F_p^r$ for $v\in K$  be the same as $f$  understood with $\mod p$ coordinates. By Hadamard's Maximum Determinant Problem, the maximal determinant of a $r\times r$ $(0, 1)$-matrix is not greater than $\frac{(r+1)^\frac{r+1}{2}}{2^r}$. For $p$ greater than this number, $f_p$ is a nondegenerate map into $X(\F_p^r)$. It follows from the fact that any simplex of $X (\F_2^r)$ is part of a base, which is also a base of $\F_p^r$ since its determinant is a nonzero number smaller than $p$.
\end{proof}

\begin{prop} If $\dim K\leq 2$ then $s_2 (K)\leq s_p (K)$ for any prime $p$.
\end{prop}

\begin{proof} We prove that if there is a nondegenerate simplicial map $f\colon K \rightarrow X(\F_2^m)$ then there is a non degenerate simplicial map $f_p\colon K \rightarrow X(\F_p^m)$.

Take $f_p$ to be defined as $f$, assuming that the coordinates are in $\F_p$. If $\dim K=1$, the claim was already proved in \cite{Baralic2023}. Suppose $\dim K=2$, we have to show that for any  2-simplex $\{v_1, v_2, v_3\}\in K$, the vectors $ f_p (v_1), f_p (v_2)$ and $f_p(v_3)$ are linearly independent over $\F_p$. Then exist $\lambda_1, \lambda_2 \in \F_p\setminus \{0\}$ such that $f_p (v_3)=\lambda_1 f_p (v_1)+\lambda_2 f_p (v_2)$.
Since  $ f_p (v_1)\neq f_p (v_2)$ there exists $i$ such that  ${f_p (v_1)}_i\neq {f_p (v_2)}_i$. We can assume that ${f_p (v_1)}_i=1,{f_p (v_2)}_i=0$ so $\lambda_1=1$. If there is $j$ such that ${f_p (v_1)}_j=0,{f_p (v_2)}_j=1$ then $\lambda_2=1$ and since $f_p(v_3)$ has the coordinates 0 or 1, we obtain $f(v_1)+f(v_2)=f(v_3)$. Thus, for any $j$ such that ${f_p (v_2)}_j=1$  we have ${f_p (v_1)}_j=1$   so $\lambda_2=p-1$. Therefore, $f_p (v_2)= f_p (v_1)+ f_p (v_3)$. But, since the coordinates are only zero or one, we obtain $f (v_2)= f (v_1)+ f (v_3)$. But, the last equality again contradicts the nondegeneracy of $f$.
\end{proof}

\begin{center}\textmd{\textbf{Acknowledgements} }
\end{center}

\medskip
All three authors were partially supported by the bilateral project `Discrete Morse theory and its applications' funded by the Ministry of Education and Science of the Republic of Serbia and the Ministry of Education, Science and Sport of the Republic of Slovenia as a part of bilateral cooperation between two countries (2020-2021). The second author was supported by the Slovenian Research and Innovation Agency program P1-0292 and the grant J1-4031.

\pagebreak
\begin{center}\textmd{\textbf{Data availability} }
\end{center}
\medskip
Data sharing is not applicable to this article as no datasets were generated or analyzed during the current study.

\begin{center}\textmd{\textbf{Code availability} }
\end{center}

\medskip
Not applicable.

\begin{center}\textmd{\textbf{Conflict of interest} }
\end{center}

\medskip
The authors have no relevant financial or non-financial interests to disclose.

\bibliographystyle{amsplain}
\bibliography{mybibliography}

\end{document}